\documentclass[10pt,oneside,english]{amsart}
\usepackage[T1]{fontenc}
\usepackage[latin9]{inputenc}
\usepackage{amstext}
\usepackage{amsthm}
\usepackage{amssymb}
\usepackage[dvipdfmx]{graphicx}

\makeatletter
\numberwithin{equation}{section}
\numberwithin{figure}{section}
\theoremstyle{plain}
\newtheorem{thm}{\protect\theoremname}
  \theoremstyle{plain}
  \newtheorem{prop}[thm]{\protect\propositionname}
  \theoremstyle{plain}
  \newtheorem{lem}[thm]{\protect\lemmaname}

\makeatother

\usepackage{babel}
  \providecommand{\lemmaname}{Lemma}
  \providecommand{\propositionname}{Proposition}
\providecommand{\theoremname}{Theorem}

\begin{document}

\title{a periodic solution of period two of a delay differential equation}

\author{Yukihiko Nakata}

\address{Department of Mathematics, Shimane University, 1060 Nishikawatsu-cho,
Matsue, Shimane, Japan}
\begin{abstract}
In this paper we prove that the following delay differential equation
\[
\frac{d}{dt}x(t)=rx(t)\left(1-\int_{0}^{1}x(t-s)ds\right),
\]
has a periodic solution of period two for $r>\frac{\pi^{2}}{2}$ (when
the steady state, $x=1$, is unstable). In order to find the periodic
solution, we study an integrable system of ordinary differential equations,
following the idea by Kaplan and Yorke \cite{Kaplan=000026Yorke:1974}.
The periodic solution is expressed in terms of the Jacobi elliptic
functions. 
\end{abstract}

\maketitle

\section{Introduction }

The delay differential equation 
\begin{equation}
\frac{d}{dt}z(t)=f\left(z(t-1)\right),\label{eq:dde_discrete}
\end{equation}
where $f:\mathbb{R}\to\mathbb{R}$ is a continuous function, has been
extensively studied in the literature. For a special case, $f(z)=r\left(1-e^{z}\right),\ r>0,$
the equation (\ref{eq:dde_discrete}) is referred to as the Wright
equation, named after the paper \cite{Wright:1955}. Jones investigated
the existence of a periodic solution of the Wright equation in \cite{Jones:1962,Jones:1962-1}
by the fixed-point theorem. Nussbaum then established a general fixed-point
theorem and study the existence of periodic solutions for a class
of functional differential equations in \cite{Nussbaum:1974}. See
also \cite{Lessard:2010,Jaquette:prep} and references therein for
the recent progress by a computer assisted approach.

Assuming that $f$ is an odd function, in the paper \cite{Kaplan=000026Yorke:1974},
Kaplan and Yorke constructed a periodic solution of the equation (\ref{eq:dde_discrete})
via a Hamiltonian system of ordinary differential equations. The idea
is used to investigate a periodic solution of the equation (\ref{eq:dde_discrete})
with a particular nonlinear function $f$ in \cite{Dormayer:1986}
and for a system of differential equations with distributed delay
in \cite{Azevedo:2007}. We refer the readers to the survey paper
\cite{Walther:2014} and the references therein. See also Chapter
XV of \cite{Diekmann:1995}. In this paper we follow the approach
by Kaplan and Yorke \cite{Kaplan=000026Yorke:1974}: we find a periodic
solution of a differential equation with distributed delay, considering
a system of ordinary differential equations.

The following mathematical model for a single species population is
known as the Hutchinson equation and as a delayed logistic equation
\begin{equation}
\frac{d}{dt}x(t)=rx(t)\left(1-x\left(t-1\right)\right).\label{eq:Huteq}
\end{equation}
The equation (\ref{eq:Huteq}) can be derived from the Wright equation
by the transformation $z(t)=\ln x(t)$. Many extension of the Hutchinson
equation (\ref{eq:Huteq}) has been investigated, see \cite{Gyori:2000,Ruan:2006,Gopalsamy:1992}
and references therein. Nevertheless, the Hutchinson-Wright equation
still poses a mathematical challenges \cite{vandenBerg:2017}.

In this paper we study the existence of a periodic solution of the
following delay differential equation
\begin{equation}
\frac{d}{dt}x(t)=rx(t)\left(1-\int_{0}^{1}x(t-s)ds\right),\label{eq:dde1}
\end{equation}
where $r$ is a positive parameter, $r>0$. The delay differential
equation (\ref{eq:dde1}) can be seen as a variant of the Hutchinson
equation (\ref{eq:Huteq}). The author's motivation to study (\ref{eq:dde1})
is that the equation appears as a limiting case of an infectious disease
model with temporary immunity (see Appendix). For the equation (\ref{eq:dde1}),
the existence of a periodic solutions does not seem to be well understood.
The periodicity, which may explain the recurrent disease dynamics,
is a trigger of this study. Differently from the discrete delay case,
the distributed delay is an obstacle, when one tries to construct
a suitable Poincare map to find a periodic solution, but see \cite{Kennedy:2014,Walther:1975,Walther:1977}.

In this paper we prove the following theorem. 
\begin{thm}
\label{thm:main}Let $r>\frac{\pi^{2}}{2}$. Then the delay differential
equation (\ref{eq:dde1}) has a periodic solution of period $2$,
i.e., $x(t)=x(t-2),\ t\in\mathbb{R}$, satisfying
\[
x(t)x(t-1)=\text{Const},\ \int_{0}^{2}x(t-s)ds=2
\]
 for any $t\in\mathbb{R}$\@.
\end{thm}

In the main text we provide more information of the periodic solution.
We prove the existence of the periodic solution, solving a corresponding
ordinary differential equation, which turns out to be equivalent to
the Duffing equation. The periodic solution, explicitly expressed
in terms of the Jacobi elliptic functions, appears at $r=\frac{\pi^{2}}{2}$,
as the positive equilibrium ($x=1$) loses stability via Hopf bifurcation.

This paper is organized as follows. In Section 2, we first study stability
of the positive equilibrium, applying the principle of linearized
stability. We then derive a system of ordinary differential equations
(\ref{eq:fourdim}) that generates the periodic solution of the original
delay differential equation (\ref{eq:dde1}). In Section 3, the system
of ordinary differential equations (\ref{eq:fourdim}) is reduced
to a scalar differential equation (\ref{eq:DuffEq}) that turns out
to be the Duffing equation. The equation is explicitly solved using
the Jacobi elliptic functions. In Section 4, we consider a fixed point
problem to find a parameter such that the period of the solution becomes
two.

\section{Preliminary}

For the delay differential equation (\ref{eq:dde1}) the natural phase
space is $C=C\left(\left[-1,0\right],\mathbb{R}\right)$ equipped
with the supremum norm (\cite{Diekmann:1995,Hale:1993}). We consider
the following initial condition 
\[
x(\theta)=\phi\left(\theta\right),\ \theta\in\left[-1,0\right]
\]
where $\phi\in C$ with $\phi(0)>0$, so that positive solutions are
generated.

Observe that (\ref{eq:dde1}) is equivalent to the following system
of delay differential equations \begin{subequations}\label{eq:twodim}
\begin{align}
\frac{d}{dt}x(t) & =-rx(t)y(t),\label{eq:twodim1}\\
\frac{d}{dt}y(t) & =x(t)-x(t-1)\label{eq:twodim2}
\end{align}
\end{subequations}with the following initial condition
\begin{align*}
x(\theta) & =\phi\left(\theta\right),\ \theta\in\left[-1,0\right]\\
y(0) & =\int_{0}^{1}\phi(-s)ds-1.
\end{align*}

It is easy to see that $x=1$ is the unique positive equilibrium for
the equation (\ref{eq:dde1}). One obtains the following result for
stability of the positive equilibrium (see also Theorem 4.1 of \cite{Oliveira:1985}).
\begin{prop}
The positive equilibrium $x=1$ is asymptotically stable for $0<r<\frac{\pi^{2}}{2}$
and unstable for $r>\frac{\pi^{2}}{2}$. Hopf bifurcation occurs at
$r=\frac{\pi^{2}}{2}$ and a periodic solution appears. 
\end{prop}

\begin{proof}
We deduce the following characteristic equation (\cite{Diekmann:1995,Hale:1993})
\begin{equation}
\lambda=-r\int_{0}^{1}e^{-\lambda s}ds,\ \lambda\in\mathbb{C}.\label{eq:cheq}
\end{equation}
Let $\lambda=\mu+i\omega,\ (\mu,\omega\in\mathbb{R})$ to obtain the
following two equations\begin{subequations}\label{eq:ch}
\begin{align}
\mu & =-r\int_{0}^{1}e^{-\mu s}\cos\left(\omega s\right)ds,\label{eq:ch1}\\
\omega & =r\int_{0}^{1}e^{-\mu s}\sin\left(\omega s\right)ds.\label{eq:ch2}
\end{align}
\end{subequations}First one sees that if $\text{Re}\lambda>0$ then
\begin{equation}
\left|\lambda\right|=\sqrt{\mu^{2}+\omega^{2}}\leq r\label{eq:estlam}
\end{equation}
Assume that there is a root in the right half complex plane (i.e.,
$\mu>0$) for sufficiently small $r>0$. One sees $\int_{0}^{1}e^{-\mu s}\cos\left(\omega s\right)ds>0$
from the estimation (\ref{eq:estlam}), thus, if $r>0$ is sufficiently
small, from (\ref{eq:ch1}) all roots of the characteristic equation
(\ref{eq:cheq}) are in the left half complex plane. 

Suppose now that for some $r>0$ purely imaginary roots exist. Substituting
$\mu=0$ into the equation (\ref{eq:ch1}), one sees that for $r=\frac{1}{2}\left(\left(2n+1\right)\pi\right)^{2}$
the characteristic equation (\ref{eq:cheq}) has purely imaginary
roots $\lambda=\pm i\omega=\pm i\left(2n+1\right)\pi$ for $n=0,1,2,\dots.$.
We show that, for $n=0,1,2,\dots,$ the purely imaginary roots $\lambda=\pm i\left(2n+1\right)\pi$
cross the imaginary axis transversally from left to right as $r$
increases in the neighborhood of $r=\frac{1}{2}\left(\left(2n+1\right)\pi\right)^{2}$.
Applying the implicit function theorem to the equation (\ref{eq:cheq}),
one has 
\[
\lambda^{\prime}(r)\left(1-r\int_{0}^{1}se^{-\lambda s}ds\right)+\int_{0}^{1}e^{-\lambda s}ds=0.
\]
One sees that 
\begin{align*}
\int_{0}^{1}se^{-\lambda s}ds & =-\frac{e^{-\lambda}}{\lambda}+\frac{1}{\lambda}\int_{0}^{1}e^{-\lambda s}ds=-\frac{e^{-\lambda}}{\lambda}-\frac{1}{r},\\
\int_{0}^{1}e^{-\lambda s}ds & =-\frac{\lambda}{r}.
\end{align*}
Therefore, at $r=\frac{1}{2}\left(\left(2n+1\right)\pi\right)^{2}$
and $\lambda=i\omega=i\left(2n+1\right)\pi$, it follows that
\[
\lambda^{\prime}\left(r\right)\left|_{r=\frac{1}{2}\left(\left(2n+1\right)\pi\right)^{2}}\right.=\frac{i\frac{\omega}{r}}{1+1+i\frac{r}{\omega}}\implies\text{Re}\lambda^{\prime}(r)\left|_{r=\frac{1}{2}\left(\left(2n+1\right)\pi\right)^{2}}\right.=\frac{1}{4+\left(\frac{r}{\omega}\right)^{2}}>0.
\]
From the Hopf bifurcation theorem, we obtain the conclusion.
\end{proof}
Thus a periodic solution of period $2$ emerges at $r=\frac{\pi^{2}}{2}$
and the positive equilibrium is unstable for $r>\frac{\pi^{2}}{2}$.

Now assume that there exists a periodic solution of period $2$ for
(\ref{eq:dde1}). Denote by $x^{*}(t)$ the periodic solution of (\ref{eq:dde1}),
i.e., $x^{*}(t)=x^{*}(t-2)$. We let 
\begin{align*}
x_{1}(t) & =x^{*}(t),\ y_{1}(t)=\int_{0}^{1}x^{*}(t-s)ds-1,\\
x_{2}(t) & =x^{*}\left(t-1\right),\ y_{2}(t)=\int_{1}^{2}x^{*}(t-s)ds-1.
\end{align*}
We are also interested in the positive periodic solution. The periodic
solution satisfies the following system of ordinary differential equations
\begin{subequations}\label{eq:fourdim}

\begin{align}
\frac{d}{dt}x_{1}(t) & =-rx_{1}(t)y_{1}(t),\label{eq:ode1}\\
\frac{d}{dt}y_{1}(t) & =x_{1}(t)-x_{2}(t),\label{eq:ode2}\\
\frac{d}{dt}x_{2}(t) & =-rx_{2}(t)y_{2}(t),\label{eq:ode3}\\
\frac{d}{dt}y_{2}(t) & =x_{2}(t)-x_{1}(t).\label{eq:ode4}
\end{align}
\end{subequations}The initial condition is \begin{subequations}\label{eq:IC_DDE}
\begin{align}
 & x_{1}(0)=a>0,\ x_{2}(0)=b>0,\label{eq:IC1}\\
 & y_{1}\left(0\right)=y_{2}\left(0\right)=0,\label{eq:IC2}
\end{align}
\end{subequations}where $a$ and $b$ will be determined later ($a=x^{*}(0)=x^{*}(2),\ b=x^{*}(-1)=x^{*}(1))$
in Section 4, so that $x_{1}(t)=x_{1}(t+2)$ holds.

From (\ref{eq:fourdim}) one sees that \begin{subequations}\label{eq:inv}
\begin{align}
y_{1}(t)+y_{2}(t) & =0,\label{eq:inv1}\\
x_{1}(t)x_{2}(t) & =ab\label{eq:inv2}
\end{align}
\end{subequations}hold for any $t\geq0$ . Thus one sees that the
periodic solution satisfies the following properties
\begin{equation}
\int_{0}^{2}x^{*}(t-s)ds=2,\ x^{*}(t)x^{*}(t-1)=\text{Const},\ t\in\mathbb{R}.\label{eq:perprop}
\end{equation}

\section{Integrable Ordinary differential equations}

We construct an initial function $\phi$ for the delay differential
equation (\ref{eq:dde1}) that generates a periodic solution of period
$2$. From (\ref{eq:fourdim}) and (\ref{eq:inv}), the system (\ref{eq:fourdim})
is reduced to the following system of ordinary differential equations\begin{subequations}\label{eq:ODE3}
\begin{align}
\frac{d}{dt}x(t) & =-rx(t)y(t),\label{eq:xeq_3}\\
\frac{d}{dt}y(t) & =x(t)-ab\frac{1}{x(t)},\label{eq:yeq_3}
\end{align}
\end{subequations}dropping the indices from $x_{1}$ and $y_{1}$
(cf. (\ref{eq:twodim})). From (\ref{eq:IC_DDE}) the initial condition
of (\ref{eq:ODE3}) is \begin{subequations}\label{eq:ODE3IC} 
\begin{align}
x(0) & =a,\label{eq:IC3}\\
y\left(0\right) & =0.\label{eq:IC4}
\end{align}
\end{subequations}We show that the system (\ref{eq:ODE3}) has a
conservative quantity.
\begin{prop}
\label{prop:Conservation}It holds that 
\begin{equation}
x(t)+ab\frac{1}{x(t)}+\frac{r}{2}y^{2}(t)=a+b,\ t\in\mathbb{R}\label{eq:cons}
\end{equation}
for the solution of the equation (\ref{eq:ODE3}) with the initial
condition (\ref{eq:ODE3IC}).
\end{prop}

\begin{proof}
Differentiating the left hand side of (\ref{eq:cons}), we obtain
\begin{align*}
 & \frac{d}{dt}\left(x(t)+ab\frac{1}{x(t)}+\frac{r}{2}y^{2}(t)\right)\\
= & \left(1-ab\frac{1}{x^{2}(t)}\right)x'(t)+ry(t)y'(t)\\
= & \left(1-ab\frac{1}{x^{2}(t)}\right)\left(-rx(t)y(t)\right)+ry(t)\left(x(t)-ab\frac{1}{x(t)}\right)\\
= & 0.
\end{align*}
From (\ref{eq:ODE3IC}), it then follows that
\[
x(t)+ab\frac{1}{x(t)}+\frac{r}{2}y^{2}(t)=x(0)+ab\frac{1}{x(0)}+\frac{r}{2}y^{2}(0)=a+b
\]
for $t\in\mathbb{R}$.
\end{proof}
Differentiating the both sides of the equation (\ref{eq:yeq_3}),
we obtain 
\begin{align*}
\frac{d^{2}}{dt^{2}}y(t) & =-\left(1+ab\frac{1}{x^{2}(t)}\right)rx(t)y(t)\\
 & =-ry(t)\left(x(t)+ab\frac{1}{x(t)}\right).
\end{align*}
Using the identity (\ref{eq:cons}) in Proposition \ref{prop:Conservation},
we derive the Duffing equation:
\begin{equation}
\frac{d^{2}}{dt^{2}}y(t)=-ry(t)\left(a+b-\frac{r}{2}y^{2}(t)\right)\label{eq:DuffEq}
\end{equation}
with the following initial condition\begin{subequations}\label{eq:duffIC}
\begin{align}
y(0) & =0,\label{eq:duffIC1}\\
\frac{d}{dt}y(0) & =x(0)-ab\frac{1}{x(0)}=a-b.\label{eq:duffIC2}
\end{align}
\end{subequations}Denote by $\text{sn}$ the Jacobi elliptic sine
function \cite{Byrd:1954,Meyer:2001}. It is known that the solution
of the Duffing equation (\ref{eq:DuffEq}) is given by 
\begin{equation}
y(t)=\alpha\text{sn}(\beta t,k),\label{eq:alphasn}
\end{equation}
where $\alpha,\beta$ and $k$ are functions of $a$ and $b$ defined
by
\begin{align}
 & \alpha(a,b)=\sqrt{\frac{2}{r}}\left(\sqrt{a}-\sqrt{b}\right),\ \beta(a,b)=\sqrt{\frac{r}{2}}\left(\sqrt{a}+\sqrt{b}\right),\label{eq:def_abk}\\
 & k(a,b)=\frac{\sqrt{a}-\sqrt{b}}{\sqrt{a}+\sqrt{b}},\label{eq:def_abk2}
\end{align}
see e.g. Chapter 4 in \cite{Kovacic} and Chapter 2 in \cite{Rand:2012}.
To simplify the notation, we occasionally drop $(a,b)$ from $\alpha,\ \beta$
and $k$.

We then obtain the exact solution of the system (\ref{eq:ODE3}) with
the initial condition (\ref{eq:ODE3IC}).
\begin{prop}
\label{prop:ODE_sol}The solution of the equations (\ref{eq:ODE3})
with the initial condition (\ref{eq:ODE3IC}) is expressed as 
\begin{align}
x(t) & =a\left(\frac{1-k}{\text{dn}(\beta t,k)-k\text{cn}(\beta t,k)}\right)^{2}=a\left(\frac{\text{dn}(\beta t,k)+k\text{cn}(\beta t,k)}{1+k}\right)^{2},\label{eq:x_cl}\\
y(t) & =\alpha\text{sn}(\beta t,k),\label{eq:y_cl}
\end{align}
where $\alpha,\beta$ and $k$ are defined in (\ref{eq:def_abk})
and (\ref{eq:def_abk2}).
\end{prop}

\begin{proof}
Since (\ref{eq:y_cl}) is given in (\ref{eq:alphasn}), we show the
equality in (\ref{eq:x_cl}), integrating the equation (\ref{eq:xeq_3}).
We get 
\[
x(t)=ae^{-r\int_{0}^{t}y(s)ds}.
\]
Using (\ref{eq:y_cl}) we compute 
\begin{align*}
r\int_{0}^{t}y(s)ds & =r\alpha\int_{0}^{t}\text{sn}(\beta u,k)du\\
 & =\frac{r\alpha}{\beta k}\left[\ln\left(\text{dn}(s,k)-k\text{cn}(s,k)\right)\right]_{0}^{\beta t}\\
 & =\frac{r\alpha}{\beta k}\ln\left(\frac{\text{dn}(\beta t,k)-k\text{cn}(\beta t,k)}{1-k}\right).
\end{align*}
Note that $\frac{r\alpha}{\beta k}=2$ holds from the definitions
in (\ref{eq:def_abk}) and (\ref{eq:def_abk2}). We then get 
\[
r\int_{0}^{t}y(s)ds=2\ln\left(\frac{\text{dn}(\beta t,k)-k\text{cn}(\beta t,k)}{1-k}\right),
\]
from which the first equality in (\ref{eq:x_cl}) follows.

Using the properties of the elliptic functions, it holds that 
\[
\left(\text{dn}-k\text{cn}\right)\left(\text{dn}+k\text{cn}\right)=\text{dn}^{2}-k^{2}\text{cn}^{2}=1-k^{2}.
\]
Therefore, we obtain the following equality
\[
\left(\frac{1-k}{\text{dn}(\beta t,k)-k\text{cn}(\beta t,k)}\right)^{2}=\left(\frac{\text{dn}(\beta t,k)+k\text{cn}(\beta t,k)}{1+k}\right)^{2}.
\]
\end{proof}

\section{Periodic solution of period $2$}

In this section we will determine $a$, the initial value for the
$x$ component of the system (\ref{eq:ODE3}), so that, for the solution
given in Proposition \ref{prop:ODE_sol}, the period is $2$ and the
integral constant becomes $-1$. The periodic solution finally solves
the delay differential equation (\ref{eq:dde1}).

Let us introduce the complete elliptic integrals of the first kind
and of the second kind \cite{Byrd:1954,Meyer:2001}. Those are respectively
given as 
\begin{align*}
K(k) & =\int_{0}^{\frac{\pi}{2}}\frac{1}{\sqrt{1-k^{2}\sin^{2}\theta}}d\theta,\\
E(k) & =\int_{0}^{\frac{\pi}{2}}\sqrt{1-k^{2}\sin^{2}\theta}d\theta
\end{align*}
for $0\leq k<1$. In the following theorem we derive the two conditions
for that the period of the solution given in Proposition \ref{prop:ODE_sol}
is $2$. 
\begin{thm}
\label{thm:twocond}Assume that the following two conditions hold
\begin{align}
 & \sqrt{\frac{r}{2}}\left(\sqrt{a}+\sqrt{b}\right)=2K(k),\label{eq:1stCondition}\\
 & \left(\sqrt{a}+\sqrt{b}\right)\sqrt{\frac{2}{r}}E(k)-\sqrt{ab}=1.\label{eq:2ndCondition}
\end{align}
Then, for the solution of the equation (\ref{eq:ODE3}) with the initial
condition (\ref{eq:ODE3IC}), it holds that
\begin{equation}
\left(x(t),y\left(t\right)\right)=\left(x(t+2),y(t+2)\right)\label{eq:periodtwo}
\end{equation}
and that
\begin{equation}
y(t)=\int_{t-1}^{t}x(s)ds-1\label{eq:integral_constant}
\end{equation}
for any $t\in\mathbb{R}$. 
\end{thm}

\begin{proof}
From (\ref{eq:1stCondition}), we have $2\beta=4K(k)$. Since the
Jacobi elliptic functions, $\text{sn},\ \text{cn}$ and $\text{dn}$,
have a period $4K(k)$, one has 
\begin{align*}
\text{sn}\left(\beta t,k\right) & =\text{sn}(\beta\left(t+2\right),k),\\
\text{cn}\left(\beta t,k\right) & =\text{cn}(\beta\left(t+2\right),k),\\
\text{dn}\left(\beta t,k\right) & =\text{dn}(\beta\left(t+2\right),k).
\end{align*}
Then it is easy to see that (\ref{eq:periodtwo}) follows from (\ref{eq:x_cl})
and (\ref{eq:y_cl}). Next we show that (\ref{eq:integral_constant})
holds. From the symmetry of the Jacobi elliptic functions, we have
\begin{align*}
\text{cn}\left(\beta\left(t-1\right),k\right) & =-\text{cn}(\beta t,k),\\
\text{dn}\left(\beta\left(t-1\right),k\right) & =\text{dn}(\beta t,k).
\end{align*}
Thus from (\ref{eq:x_cl}) we obtain
\[
x(t-1)=a\left(\frac{\text{dn}(\beta\left(t-1\right),k)+k\text{cn}(\beta\left(t-1\right),k)}{1+k}\right)^{2}=a\left(\frac{\text{dn}\left(\beta t,k\right)-k\text{cn}\left(\beta t,k\right)}{1+k}\right)^{2}
\]
and $x(t)x(t-1)=a^{2}\left(\frac{1-k}{1+k}\right)^{2}=ab$ follows.
Then from (\ref{eq:yeq_3}), for the solution of the equation (\ref{eq:ODE3}),
we have the following equality
\[
\frac{d}{dt}y(t)=x(t)-x(t-1),
\]
implying that 
\begin{equation}
y(t)=\int_{0}^{1}x(t-s)ds+\text{const}.\label{eq:intconst}
\end{equation}
From (\ref{eq:1stCondition}) (i.e., $2\beta=4K(k)$) and (\ref{eq:y_cl})
we have 
\[
y(0)=y(1)=0.
\]
Now we show that 
\begin{equation}
\int_{0}^{1}x(t)dt=\left(\sqrt{a}+\sqrt{b}\right)\sqrt{\frac{r}{2}}E(k)-\sqrt{ab}.\label{eq:integral_xt}
\end{equation}
Using the properties of the Jacobi elliptic functions \cite{Byrd:1954},
we compute
\begin{align*}
\left(\text{dn}+k\text{cn}\right)^{2} & =\text{dn}^{2}+2k\text{dn}\text{cn}+k^{2}\text{cn}^{2}=k^{2}-1+2\text{dn}^{2}+2k\text{dn}\text{cn}.
\end{align*}
From the following computations
\begin{align*}
\int_{0}^{1}\text{dn}^{2}(\beta t,k)dt & =\frac{1}{\beta}\int_{0}^{2K(k)}\text{dn}^{2}(u,k)dt=\frac{2}{\beta}E(k),\\
\int_{0}^{1}\text{dn}(\beta t,k)\text{cn}(\beta t,k)dt & =0,
\end{align*}
one sees that 
\[
\int_{0}^{1}x(t)dt=a\frac{k^{2}-1+\frac{4}{\beta}E(k)}{\left(1+k\right)^{2}}.
\]
by (\ref{eq:x_cl}). Then we obtain (\ref{eq:integral_xt}) from (\ref{eq:def_abk})
and (\ref{eq:def_abk2}). from the condition (\ref{eq:2ndCondition})
the integral constant in (\ref{eq:intconst}) becomes $-1$, for the
solution of the equation (\ref{eq:ODE3}). 
\end{proof}
The conditions (\ref{eq:1stCondition}) and (\ref{eq:2ndCondition})
ensure the existence of a periodic solution of period 2 for the system
of ordinary differential equations (\ref{eq:ODE3}), satisfying (\ref{eq:integral_constant}).
The periodic solution obtained in Theorem \ref{thm:twocond} is also
a periodic solution of the delay differential equation (\ref{eq:dde1}).
Our remaining task is to interpret the conditions (\ref{eq:1stCondition})
and (\ref{eq:2ndCondition}) in terms of the parameter $r$ in the
equation (\ref{eq:dde1}).

Eliminating $a$ and $b$ from the conditions (\ref{eq:1stCondition})
and (\ref{eq:2ndCondition}), we obtain the following equality
\begin{equation}
r=L(k),\ 0\leq k<1,\label{eq:Leq}
\end{equation}
where 
\[
L(k):=2K(k)\left(2E(k)-K(k)\left(1-k^{2}\right)\right).
\]
For the derivation of (\ref{eq:Leq}), see the proof of Proposition
\ref{prop:ab} below. Now we show that the equation (\ref{eq:Leq})
has a unique root.
\begin{lem}
\label{lem:L}The function $L$ is a strictly increasing function
with 
\[
L(0)=\frac{\pi^{2}}{2}<\lim_{k\to1-0}L(k)=\infty.
\]
\end{lem}

\begin{proof}
From the definition of $L$, it is easy to see $L(0)=\frac{\pi^{2}}{2}.$
By the straightforward calculation, we obtain
\begin{align*}
 & \frac{d}{dk}\left(2E(k)-K(k)\left(1-k^{2}\right)\right)\\
= & \frac{2}{k}\left(E\left(k\right)-K\left(k\right)\right)-\frac{1}{k}\left(E\left(k\right)-\left(1-k^{2}\right)K\left(k\right)\right)+2kK\left(k\right)\\
= & \frac{1}{k}\left(E\left(k\right)-\left(1-k^{2}\right)K\left(k\right)\right)\\
= & \left(1-k^{2}\right)K^{\prime}\left(k\right)\\
> & 0,
\end{align*}
noting that 
\begin{align*}
K^{\prime}(k) & =\frac{1}{k}\left(\frac{E(k)}{1-k^{2}}-K(k)\right).\\
E^{\prime}(k) & =\frac{1}{k}\left(E(k)-K(k)\right),
\end{align*}
see e.g. P. 282 of \cite{Byrd:1954}. Since it can be shown that 
\[
\lim_{k\to1-0}K(k)(1-k^{2})=\lim_{k\to1-0}\int_{0}^{\frac{\pi}{2}}\frac{1-k^{2}}{\sqrt{1-k^{2}\sin^{2}\theta}}d\theta=0,
\]
$L$ is a strictly increasing function with $\lim_{k\to1-0}L(k)=\infty$. 
\end{proof}
Then, $a$ and $b$ are determined by the following Proposition.
\begin{prop}
\label{prop:ab}There exist $a>0$ and $b>0$ such that the two conditions
(\ref{eq:1stCondition}) and (\ref{eq:2ndCondition}) in Theorem \ref{thm:twocond}
hold if and only if $r>\frac{\pi^{2}}{2}$. In particular, $a$ and
$b$ are given as
\begin{equation}
\left[\begin{array}{c}
a\\
b
\end{array}\right]=\frac{K(k)}{2E(k)-K(k)\left(1-k^{2}\right)}\left[\begin{array}{c}
\left(1+k\right)^{2}\\
\left(1-k\right)^{2}
\end{array}\right],\label{eq:ab}
\end{equation}
where $k=L^{-1}(r),\ r>\frac{\pi^{2}}{2}$.
\end{prop}

\begin{proof}
Consider $a>0$ and $b>0$ for the two equations (\ref{eq:1stCondition})
and (\ref{eq:2ndCondition}). From the definition of $k$ in (\ref{eq:def_abk2})
we have 
\begin{equation}
\sqrt{b}=\frac{1-k}{1+k}\sqrt{a}.\label{eq:a1}
\end{equation}
thus the two conditions (\ref{eq:1stCondition}) and (\ref{eq:2ndCondition})
are expressed in terms of $a$ and $k$, namely 
\begin{equation}
\sqrt{a}=K(k)\left(1+k\right)\sqrt{\frac{2}{r}}.\label{eq:a2}
\end{equation}
Substituting (\ref{eq:a2}) to (\ref{eq:2ndCondition}), we arrive
at the following equation 
\[
L(k)=r,\ 0\leq k<1.
\]
From Lemma \ref{lem:L}, for $r>\frac{\pi^{2}}{2}$, we can find $k=L^{-1}\left(r\right)>0$.
From (\ref{eq:a1}) and (\ref{eq:a2}), $a$ and $b$ can be computed
as in (\ref{eq:ab}).
\end{proof}
Finally we obtain the following theorem
\begin{thm}
\label{thm:maintheo}Let $r>\frac{\pi^{2}}{2}$. Then the delay differential
equation (\ref{eq:dde1}) has a periodic solution of period $2$.
The periodic solution is expressed as in (\ref{eq:x_cl}), where $a$
and $b$ are determined in Proposition \ref{prop:ab}.
\end{thm}

Denote by $x^{*}(t)$ the periodic solution of (\ref{eq:dde1}) with
$x^{*}(0)=a$, which satisfies (\ref{eq:perprop}). It is easy to
see that 
\[
\max_{t\in\left[-1,1\right]}x^{*}(t)=x^{*}(0)=a,\ \min_{t\in\left[-1,1\right]}x^{*}(t)=x^{*}(-1)=x^{*}(1)=b.
\]
Thus one sees that $x^{*}(t)x^{*}(t-1)=ab$ for $t\in\mathbb{R}$.
From (\ref{eq:ab}) it can be shown that $\lim_{r\to\infty}\left(a,b\right)=\left(\infty,0\right)$,
thus the amplitude of the periodic solution tends to $\infty$ as
$r\to\infty$. We also note that 
\[
\lim_{r\to\infty}ab=0.
\]
Finally, from the symmetry of the Jacobi elliptic functions, it follows
that 
\[
x^{*}(2n+s)=x^{*}(2n-s),\ n\in\mathbb{Z},\ s\in\mathbb{R}.
\]

\begin{figure}
\begin{centering}
\includegraphics[scale=0.8]{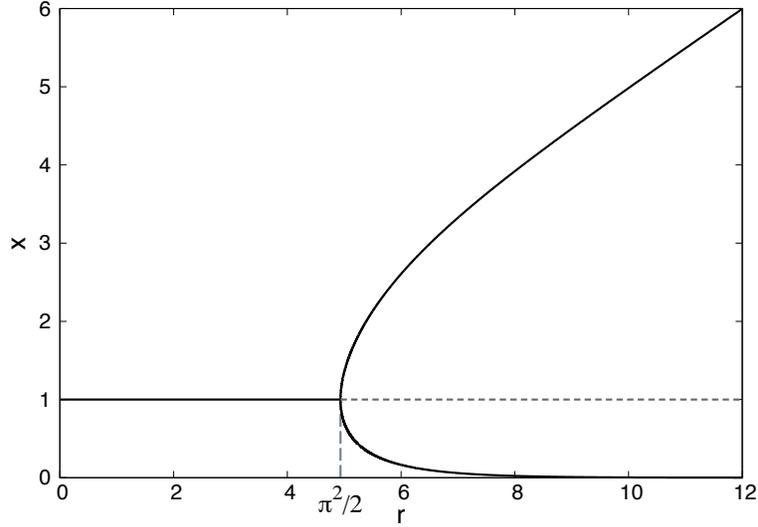}
\par\end{centering}
\caption{Bifurcation of the equilibrium. The equilibrium $x=1$ is asymptotically
stable for $r<\frac{\pi^{2}}{2}$ and is unstable for $r>\frac{\pi^{2}}{2}$.
At $r=\frac{\pi^{2}}{2}$ a Hopf bifurcation occurs and the periodic
solution appears.}
\end{figure}

\begin{figure}
\includegraphics[scale=0.8]{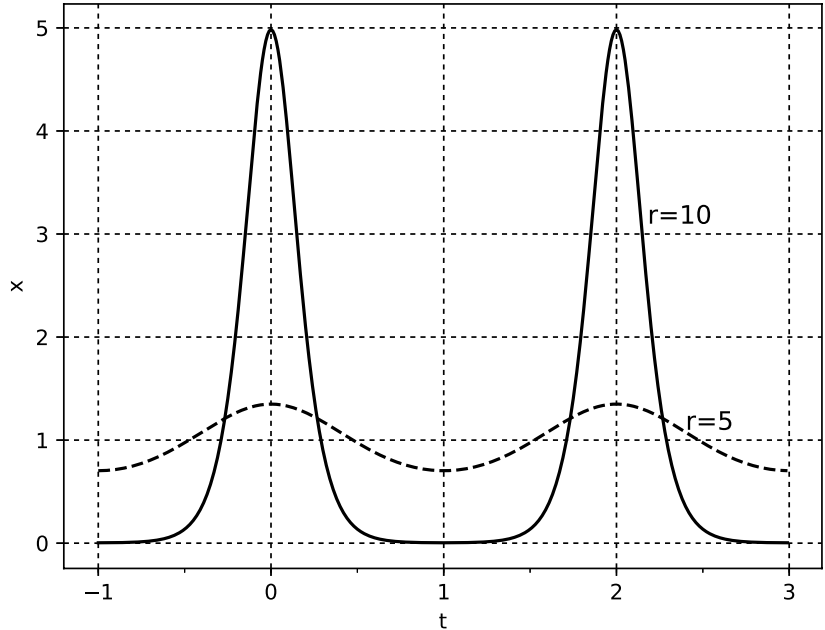}\caption{Time profile of the periodic solution for $r=5$ and $r=10$.}

\end{figure}

\section{Discussion}

In this paper we prove the existence of a periodic solution of the
delay differential equation (\ref{eq:dde1}). The periodic solution
satisfies the nonlinear ordinary differential equation (\ref{eq:fourdim}).
Since the system (\ref{eq:fourdim}) has conservative quantities,
the system (\ref{eq:fourdim}) is reduced to the Duffing equation
with surprise. Then we obtain the explicit solution in terms of the
Jacobi elliptic functions. 

Primarily, with Gabor Kiss and Gabriella Vas, the project has started
with finding a periodic solution of the delay differential equation
of the form
\begin{equation}
\frac{d}{dt}x(t)=rx(t)\left(1-\alpha x(t)-\int_{0}^{1}x(t-s)ds\right),\label{eq:inst}
\end{equation}
where $0\leq\alpha<1$. The delay differential equation (\ref{eq:inst})
arises from a mathematical model for disease transmission dynamics,
as it is explained in Appendix. Differently from the Wright equation,
the estimation of the non-delay term, together with the distributed
delay term, seems to be an obstacle, when one tries to construct a
Poincare map to find a periodic solution (cf. \cite{Kennedy:2014,Walther:1975,Walther:1977}).
Multiple periodic solutions seem to be possible for the SIRS model
in Appendix with the demographic turn-over \cite{TaylorCarr:2009}.
Multiplicity of the periodic solution is also shown for logistic equations
with multiple delays \cite{KissLessard:2012}. 

In this paper we study the delay differential equation (\ref{eq:dde1}),
setting $\alpha=0$ in the equation (\ref{eq:inst}) to simplify the
problem. Numerically we observe the periodic solution of period $2$
that attracts many positive solutions. Uniqueness and stability of
the periodic solution is an open problem, which are left for a future
work.

\subsection*{Acknowledgement}

The work has been initiated at the discussion with Prof. Hans-Otto
Walther, who kindly introduced his habilitation thesis to the author.
The author is grateful for his hospitality at the University of Giessen
in February 2016. The author thanks Gabriella Vas and Gabor Kiss for
a lot of discussions on the periodic solutions of delay differential
equations. The author also thanks Prof. Benjamin Kennedy and Prof.
Tibor Krisztine for the interest in the study. Finally the author
would like to thank Prof. Emiko Ishiwata, who kindly introduced the
area of integrable systems to the author. The author was supported
by JSPS Grant-in-Aid for Young Scientists (B) 16K20976 of Japan Society
for the Promotion of Science.

\appendix

\section{Appendix}

The delay differential equation (\ref{eq:dde1}) can be related to
an epidemic model that accounts for temporary immunity (\cite{BlyussKyrychko:2010,HethcoteStech=000026Driessche:1981,Goncaleves:2010,TaylorCarr:2009,Yuan:2014}).
Let us derive the delay differential equation (\ref{eq:dde1}) as
a limiting case of the following SIRS type epidemic model with temporary
immunity \begin{subequations}\label{eq:SIRS}
\begin{align}
\frac{d}{dt}S(t) & =-\beta S(t)I(t)+\gamma I(t-\tau),\label{eq:SIRS1}\\
\frac{d}{dt}I(t) & =\beta S(t)I(t)-\gamma I(t),\label{eq:SIRS2}\\
\frac{d}{dt}R(t) & =\gamma I(t)-\gamma I(t-\tau).\label{eq:SIRS3}
\end{align}
\end{subequations}where $S(t),\ I(t)$ and $R(t)$ respectively denote
the fraction of susceptible, infective and recovered populations at
time $t$. The model describes transition of susceptible, infective
and recovered populations. The model (\ref{eq:SIRS}) has three parameters:
transmission coefficient $\beta>0$, the recovery rate $\gamma>0$
and the immune period $\tau>0$. The initial condition is given as
follows 
\begin{align*}
S(0) & =S_{0}>0\\
I(s) & =\psi(s),\ s\in\left[-\tau,0\right]\\
R(0) & =\gamma\int_{0}^{\tau}\psi(-s)ds,
\end{align*}
where $\psi$ is a positive continuous function. We now require that
\[
S(0)+I(0)+R(0)=S_{0}+\psi(0)+\gamma\int_{0}^{\tau}\psi(-s)ds=1,
\]
so that 

\begin{equation}
S(t)+I(t)+R(t)=1,\ t\geq0\label{eq:Total}
\end{equation}
implying that the total population is constant. We also have the following
identity 
\begin{equation}
R(t)=\gamma\int_{0}^{\tau}I(t-s)ds,\ t\geq0.\label{eq:Recovered}
\end{equation}

From (\ref{eq:Total}) and (\ref{eq:Recovered}) we get 
\[
S(t)=1-I(t)-\gamma\int_{0}^{\tau}I(t-s)ds.
\]
Then from (\ref{eq:SIRS2}) we obtain the following scalar delay differential
equation 
\begin{equation}
\frac{d}{dt}I(t)=I(t)\left\{ \beta\left(1-I(t)-\gamma\int_{0}^{\tau}I(t-s)ds\right)-\gamma\right\} .\label{eq:Ieq}
\end{equation}
We let $x(t)=\frac{I(t)}{I_{e}},$ where $I_{e}$ is a nontrivial
equilibrium of (\ref{eq:Ieq}) given as 
\[
I_{e}=\frac{1-\frac{\gamma}{\beta}}{1+\gamma\tau}.
\]
It is assumed that $\beta>\gamma$ to ensure $I_{e}>0$. Considering
a nondimensional time so that the immune period is $1$, we obtain
\[
\frac{d}{dt}x(t)=\left(\beta-\gamma\right)x(t)\left(1-\frac{x(t)+\gamma\tau\int_{0}^{1}x(t-s)ds}{1+\gamma\tau}\right).
\]
We now fix $r=\beta-\gamma$ and let $\gamma\tau\to\infty$ to formally
obtain the equation (\ref{eq:dde1}). Local stability analysis for
(\ref{eq:Ieq}) can be found in \cite{HethcoteStech=000026Driessche:1981,Goncaleves:2010}.
See also \cite{Omori:2015} for the application of the mathematical
model to explain the periodic outbreak of a childhood disease.

\begin{thebibliography}{10}
\bibitem{Azevedo:2007}Azevedo, Katia A.G., Gadotti, Marta C., Ladeira,
Luiz A.C., Special symmetric periodic solutions of differential systems
with distributed delay. Nonlinear Anal. 67 (2007), no. 6, 1861\textendash 1869. 

\bibitem{BlyussKyrychko:2010}Blyuss, K.B., Kyrychko, Y.N., Stability
and bifurcations in an epidemic model with varying immunity period.
B. Math. Biol. 72, 490\textendash 505 (2010)

\bibitem{Byrd:1954}Byrd, P.F., Friedman, M.D., Handbook of Elliptic
Integrals for Engineers and Physicists. Springer Verlag, (1954)

\bibitem{Diekmann:1995}Diekmann, O., van Gils, S.A., Verduyn Lunel,
S.M., Walther, H.O., Delay Equations: Functional-, Complex- and Nonlinear
Analysis. Springer, New York (1995)

\bibitem{Dormayer:1986}Dormayer, P, Exact fomulae for periodic solutions
of $x^{\prime}(t+1)=\alpha(-x(t)+bx^{3}(t))$. J. Appl. Math. Phys.
37, 765\textendash 775 (1986)

\bibitem{Gopalsamy:1992}Gopalsamy, K., Stability and oscillations
in delay differential equations of population dynamics. Mathematics
and its Applications, 74. Kluwer Academic Publishers Group, Dordrecht
(1992)

\bibitem{Goncaleves:2010}Gon\c{c}alves, S., Guillermo, A., Gomes,
M.F.C., Oscillations in SIRS model with distributed delays. The European
Physical Journal B-Condensed Matter and Complex Systems 81.3 (2011)
363-371.

\bibitem{Gyori:2000}Gy\"ori, I. A new approach to the global asymptotic
stability problem in a delay Lotka-Volterra differential equation.
Mathematical and computer modelling 31.6-7 (2000): 9-28.

\bibitem{Hale:1993}Hale, J.K., Verduyn Lunel, S.M., Introduction
to Functional Differential Equations. Springer, New York (1993)

\bibitem{HethcoteStech=000026Driessche:1981}Hethcote, H. W., Stech,
H. W., van den Driessche, P. Nonlinear oscillations in epidemic models.
SIAM Journal on Applied Mathematics, 40(1) (1981) 1-9.

\bibitem{Jaquette:prep}Jaquette, J., A proof of Jones' conjecture.
Preprint

\bibitem{Jones:1962}Jones, G.S,, The existence of periodic solutions
of $f^{\prime}(x)=-\alpha f(x-1)\left\{ 1+f(x)\right\} $. J. Math.
Anal. Appl. 5, 435\textendash 450 (1962) 

\bibitem{Jones:1962-1}Jones, G. S., On the nonlinear differential-difference
equation $f^{\prime}(x)=-\alpha f(x-1)\left\{ 1+f(x)\right\} $. J.
Math. Anal. Appl. 4 (1962) 440\textendash 469

\bibitem{Kaplan=000026Yorke:1974}Kaplan, J.L., Yorke, J.A., Ordinary
differential equations which yield periodic solutions of differential
delay equations. J. Math. Anal. Appl. 48, 317\textendash 324 (1974) 

\bibitem{Kennedy:2014}Kennedy, B., Symmetric periodic solutions for
a class of differential delay equations with distributed delay. Electronic
Journal of Qualitative Theory of Differential Equations 2014.4 (2014):
1\textendash 18.

\bibitem{KissLessard:2012}Kiss, G., Lessard, J.P., Computational
fixed-point theory for differential delay equations with multiple
time lags, Journal of Differential Equations 252.4 (2012): 3093\textendash 3115.

\bibitem{Kovacic}Kovacic, I., Brennan, M.J., The Duffing Equation:
Nonlinear Oscillators and Their Behaviour, John Wiley \& Sons, Chichester,
2011.

\bibitem{Lessard:2010}Lessard, J.P., Recent advances about the uniqueness
of the slowly oscillating periodic solutions of Wright's equation.
J. Differ. Equ. 248, 992\textendash 1016 (2010)

\bibitem{Meyer:2001}Meyer, K.R., Jacobi elliptic functions from a
dynamical systems point of view. The American Mathematical Monthly
108.8 (2001): 729-737.

\bibitem{Nussbaum:1974}Nussbaum, R.D., Periodic solutions of some
nonlinear autonomous functional differential equations. Annali di
Matematica Pura ed Applicata 101.1 (1974): 263-306.

\bibitem{Nussbaum:1973}Nussbaum, R.D., Periodic solutions of some
nonlinear, autonomous functional differential equations. II. Journal
of Differential Equations 14.2 (1973): 360-394.

\bibitem{Oliveira:1985}Oliveira, J.C.F.D., Carvalho, L.A.V., A Lyapunov
Functional for a Retarded Differential Equation. SIAM J. Math. Anal.,
16 (6) (1985): 1295\textendash 1305.

\bibitem{Omori:2015}Omori, R., Nakata, Y., Tessmer, H.L., Suzuki,
S., Shibayama, K., The determinant of periodicity in Mycoplasma pneumoniae
incidence: an insight from mathematical modelling. Scientific reports,
5:14473 DOI:10.1038/srep14473 (2015)

\bibitem{Rand:2012}Rand, R.H., Lecture Notes on Nonlinear Vibrations.
2012. https://ecommons.cornell.edu/handle/1813/28989

\bibitem{Ramussen:2003}Rasmussen, H., Wake G.C., Donaldson, J., Analysis
of a class of distributed delay logistic differential equations. Mathematical
and Computer Modelling 38.1-2 (2003): 123-132.

\bibitem{Ruan:2006}Ruan, S., Delay differential equations in single
species dynamics, Delay Differential Equations and Applications 205
(2006): 477-517.

\bibitem{TaylorCarr:2009}Taylor, M. L., Carr, T. W., An SIR epidemic
model with partial temporary immunity modeled with delay. Journal
of Mathematical Biology, 59(6) (2009): 841-880.

\bibitem{vandenBerg:2017}van den Berg, J.B., Jaquette, J., A proof
of Wright\textquoteright s conjecture. Preprint arXiv:1704.00029,
2017.

\bibitem{Walther:1975}Walther, H.O., Existence of a non-constant
periodic solution of a non-linear autonomous functional differential
equation representing the growth of a single species population. Journal
of Mathematical Biology, 1.3 (1975): 227-240.

\bibitem{Walther:1977}Walther, H.O., \"Uber Ejektivit\"at und periodische
L\"osungen bei autonomen Funktionaldifferentialgleichungen mit verteilter
Verz\"ogerung. Habilitattionsschrift zur Erlangung der venia legendi
f\"ur des Fach Mathematik am Fachbereich Mathematik der Ludwig-Maximilians-Universit\"at
M\"unchen (1977)

\bibitem{Walther:2014}Walther, H.O., Topics in delay differential
equations. Jahresber. Dtsch. Math.-Ver. 116 (2014), no. 2, 87\textendash 114.

\bibitem{Wright:1955}Wright, E.M., A non-linear difference-differential
equation. J. Reine Angew. Math. 194 (1955) 66\textendash 87.

\bibitem{Yuan:2014}Yuan, Y., B\'elair, J. Threshold dynamics in
an SEIRS model with latency and temporary immunity. Journal of Mathematical
Biology, 69(4) (2014) 875-904.
\end{thebibliography}
\end{document}